\documentclass[10pt]{amsart}

\usepackage{amsmath,amsthm,amsfonts,amssymb,enumerate,mathrsfs,setspace,verbatim,latexsym}

\newtheorem{thm}{Theorem}

\newtheorem{lem}[thm]{Lemma}
\newtheorem{prop}[thm]{Proposition}
\newtheorem{cor}[thm]{Corollary}
\numberwithin{thm}{section}
\numberwithin{equation}{section}

\theoremstyle{definition}

\newtheorem{conj}[thm]{Conjecture}

\newcommand{\intg}{\mathbb Z}
\newcommand{\nat}{\mathbb N}

\newcommand{\PP}{\mathcal P}

\renewcommand{\mod}[1]{{\ifmmode\text{\rm\ (mod~$#1$)}\else\discretionary{}{}{\hbox{ }}\rm(mod~$#1$)\fi}}
\newcommand{\ep}{\varepsilon}

\newcommand{\li}{\mathop{\rm li}}

\begin{document}
\title[Polynomials related to the Goldbach conjecture]{Polynomials whose reducibility is related to the Goldbach conjecture}
\author[P. Borwein \and K.K. Choi \and G. Martin \and C.L. Samuels]{Peter Borwein \and Kwok-Kwong Stephen Choi \and
Greg Martin \and Charles L. Samuels}
\address{Simon Fraser University, Department of Mathematics, 8888 University Drive, Burnaby, BC V5A 1S6, Canada}
\address{University of British Columbia, Department of Mathematics, 1984 Mathematics Road, Vancouver, BC V6T 1Z2, Canada}
\thanks{Research of all authors is supported in part by NSERC of Canada}
\date{\today}

\begin{abstract}
  We introduce a collection of polynomials $F_N$, associated to each positive integer $N$, whose divisibility properties yield a reformulation of the Goldbach conjecture.  
  While this reformulation certainly does not lead to a resolution of the conjecture, it does suggest two natural generalizations for which we provide 
  some numerical evidence.  As these polynomials $F_N$ are independently interesting, we further explore their basic properties, giving, among other things, asymptotic estimates on the 
  growth of their coefficients.
\end{abstract}

\maketitle

\section{Introduction} \label{PolyIntro}

Let $\PP$ denote the set of odd primes.  One of the oldest unsolved problems in mathematics concerns the set
$\PP + \PP = \{p + q: p,q\in \PP\}.$

\begin{conj}[Goldbach Conjecture]
      If $N > 4$ is an even integer, then $N\in\PP + \PP$.
\end{conj}

If $N$ is any positive integer, we say that the {\it Goldbach conjecture holds for $N$} if $N\in\PP + \PP$.  Otherwise, we say the {\it the Goldbach conjecture fails for $N$}.
Of course, we make no attempt here to prove the Goldbach conjecture, however we wish to study a related collection of polynomials.  In order to construct these polynomials, we let
$\chi_\PP:\nat\to\{0,1\}$ denote the indicator function of $\PP$.  That is,
\begin{equation*}
      \chi_\PP(n) = \left\{\begin{array}{ll}
          1 & \mathrm{if}\ n\ \mathrm{is\ an\ odd\ prime,} \\
          0 & \mathrm{otherwise}.
          \end{array}\right.
\end{equation*}
Furthermore, for each positive integer $N$, we define
\begin{equation*}
      R(N) = \sum_{n=1}^{N-1}\chi_\PP(n)\chi_\PP(N-n)
\end{equation*}
so that $R(N)$ counts the number of ways to write $N$ as a sum of two odd primes.  We note that $R(N) = 0$ if and only if $N\not\in\PP+\PP$.
To each positive integer $N$, we associate a polynomial $F_N\in\intg[x]$ given by
\begin{equation*}
       F_N(z) = \sum_{k=0}^{N-1} \left( \sum_{n=1}^{N-1}\chi_\PP(n)z^{kn}\right)^2.
\end{equation*}
Our first result shows that the $F_N(z)$ are closely related to the Goldbach problem.  In this article, we will always use $\Phi_N$ to denote the $N$th
cyclotomic polynomial.

\begin{thm} \label{FDivisibility}
  Suppose that $N$ is a positive integer.  Then $\Phi_N$ divides $F_N$ if and only if the Goldbach conjecture fails for $N$.
\end{thm}

In other words, Theorem \ref{FDivisibility} reformulates the Goldbach conjecture in terms of the divisibility properties of $F_N$.
Since no odd integer can be written as a sum of odd primes, we observe immediately that $\Phi_N$ divides $F_N$ for all odd $N$.
Naively, it is reasonable to conjecture that $F_N$ is irreducible for all even integers $N > 4$.  Unfortunately, $F_N$ always has at least 
one non-trivial irreducible factor.

\begin{thm} \label{2NCyclo}
  If $N$ is a positive integer then $\Phi_{2N}$ divides $F_N$.
\end{thm}

Early numerical evidence seems to suggest that $F_N/\Phi_{2N}$ is, in fact, irreducible for all even integers $N>4$.
If this is the case, then the Goldbach conjecture would follow.  Similarly, it appears that, for odd integers $N>5$, we have that $F_N/(\Phi_N\Phi_{2N})$ 
is irreducible.  Although this is not relevant to the Goldbach conjecture, we find it independently interesting.

\begin{conj} \label{GeneralGoldbach}
      If $N>5$ is an integer then the following conditions hold.
      \begin{enumerate}[(i)]
      \item\label{Evens} If $N$ is even, then $F_N/\Phi_{2N}$ is irreducible.
      \item If $N$ is odd, then $F_N/\Phi_N\Phi_{2N}$ is irreducible.
      \end{enumerate}
\end{conj}

As we have noted, Conjecture \ref{GeneralGoldbach} \eqref{Evens} would imply the Goldbach conjecture.  However, the converse is possibly false.  Indeed,
$F_N/\Phi_{2N}$ could be reducible but still not divisible by $\Phi_{N}$.  As such, we should view Conjecture \ref{GeneralGoldbach} as being significantly harder 
than the Goldbach conjecture, and therefore, not likely within reach using current techniques.  Nonetheless, we find it interesting to see the Goldbach conjecture in this context.

As evidence in favor of Conjecture \ref{GeneralGoldbach}, we have found that it holds for all $N\leq 50$.  
For even $N$, the first few polynomials $F_N/\Phi_N$ are given in the following list.
\begin{align*}
	F_6/\Phi_{12} & = {z}^{46}+{z}^{44}-{z}^{40}-{z}^{38}+3\,{z}^{36}+4\,{z}^{34}+{z}^{32}-3\,{z}^{30}-2\,{z}^{28}+3\,{z}^{26}\\
        & \qquad +5\,{z}^{24}+2\,{z}^{22} -2\,{z}^{18}-{z}^{16}+2\,{z}^{14}+5\,{z}^{12}+3\,{z}^{10}-{z}^{8}-3\,{z}^{6}+4\,{z}^{2}+4 \\
	F_8/\Phi_{16} & =  {z}^{90}-{z}^{82}+3\,{z}^{76}+{z}^{74}-3\,{z}^{68}-{z}^{66}+2\,{z}^{64}+4\,{z}^{62}+3\,{z}^{60}+{z}^{58}\\
        & \qquad -2\,{z}^{56}-4\,{z}^{54}+2\,{z}^{52}-{z}^{50}+5\,{z}^{48}+4\,{z}^{46}-2\,{z}^{44}+4\,{z}^{42}-{z}^{40}\\
        & \qquad -4\,{z}^{38}+2\,{z}^{36}-2\,{z}^{34}+6\,{z}^{32} +4\,{z}^{30}+{z}^{28}+2\,{z}^{26}-4\,{z}^{24}-2\,{z}^{18}\\
        & \qquad +9\,{z}^{16}+3\,{z}^{12}+3\,{z}^{10}-7\,{z}^{8}+{z}^{6}+9\\
	F_{10}/\Phi_{20} & = {z}^{118}+{z}^{116}-{z}^{108}-{z}^{106}+{z}^{104}+{z}^{102}+2\,{z}^{100}+3\,{z}^{98}+{z}^{96}-{z}^{94}\\
        & \qquad -{z}^{92} -{z}^{90} +{z}^{86}+{z}^{84}+4\,{z}^{82}+4\,{z}^{80}+2\,{z}^{76}+2\,{z}^{74}-{z}^{72}-{z}^{70}\\
        & \qquad -2\,{z}^{66}+2\,{z}^{64}+9\,{z}^{62}+5\,{z}^{60} +4\,{z}^{56} -4\,{z}^{52}+3\,{z}^{48}+{z}^{44}+7\,{z}^{42}+8\,{z}^{40}\\
        & \qquad +2\,{z}^{38}+{z}^{34}-3\,{z}^{30}+{z}^{28}+3\,{z}^{26}+{z}^{24}+6\,{z}^{22}+8\,{z}^{20}+2\,{z}^{16}\\
        & \qquad +4\,{z}^{14}-3\,{z}^{12}-4\,{z}^{10}+3\,{z}^{8}+{z}^{6}+9\,{z}^{2}+9.\\
\end{align*}
Now we give the analogous list but for odd $N$.
\begin{align*}
	F_{7}/(\Phi_{7}\Phi_{14}) & = {z}^{48}-{z}^{46}+{z}^{38}+{z}^{36}-{z}^{34}-{z}^{32}+3\,{z}^{28}-3\,{z}^{26}+2\,{z}^{24}\\
        & \qquad +{z}^{20}-{z}^{18} -2\,{z}^{16}+3\,{z}^{14}-{z}^{10}+{z}^{8}+{z}^{6}-4\,{z}^{2}+4\\
	F_{9}/(\Phi_{9}\Phi_{18}) & = {z}^{100}-{z}^{94}+{z}^{86}+2\,{z}^{84}+{z}^{82}-{z}^{80}-2\,{z}^{78}-{z}^{76}+3\,{z}^{72}\\
        & \qquad +4\,{z}^{68}-{z}^{66}+{z}^{64}-4\,{z}^{62}+3\,{z}^{58}+{z}^{54}-2\,{z}^{52}+4\,{z}^{50}+4\,{z}^{48}\\
        & \qquad -{z}^{46}-{z}^{44}-5\,{z}^{42}+3\,{z}^{40} +6\,{z}^{36}-2\,{z}^{34}+{z}^{32}+{z}^{30}+4\,{z}^{28}\\
        & \qquad -{z}^{26}-4\,{z}^{24}-2\,{z}^{22}+2\,{z}^{20}+7\,{z}^{18}-{z}^{16} -{z}^{14}+2\,{z}^{12}+3\,{z}^{10}\\
        & \qquad +2\,{z}^{8}-8\,{z}^{6}+9\\
	F_{11}/(\Phi_{11}\Phi_{22}) & = {z}^{120}-{z}^{118}+{z}^{106}-{z}^{104}+2\,{z}^{100}-{z}^{98}-{z}^{96}+{z}^{92}-{z}^{90}\\
        & \qquad +2\,{z}^{88}-2\,{z}^{86} +{z}^{84}-{z}^{82}+3\,{z}^{80}-3\,{z}^{74}+4\,{z}^{70}-4\,{z}^{68}+2\,{z}^{66}\\
        & \qquad +{z}^{64}-2\,{z}^{62}+4\,{z}^{60}-2\,{z}^{58} +{z}^{52}-4\,{z}^{46}+4\,{z}^{44}-{z}^{42}+4\,{z}^{40}\\
        & \qquad -2\,{z}^{38}+{z}^{36}-2\,{z}^{34}-{z}^{32}+4\,{z}^{30}+2\,{z}^{28} -5\,{z}^{26} -4\,{z}^{24}+6\,{z}^{22}\\
        & \qquad +2\,{z}^{20}-{z}^{18}+{z}^{16}-2\,{z}^{14}+{z}^{10}+{z}^{8}+{z}^{6}-9\,{z}^{2}+9.\\
\end{align*}
Indeed, we have found that the right hand sides on the above lists are all irreducible over $\intg$.

Because of their relevance to the Goldbach conjecture, it may also be interesting to study the number of roots of $F_N$ that lie on the unit circle. 
In view of Theorem \ref{2NCyclo}, it is clear that $F_N$ has at least $\varphi(2N)$ such roots.   For even integers $N>4$, if $F_N$ has
no other roots on the unit circle, then the Goldbach conjecture would follow from Theorem \ref{FDivisibility}.  Our numerical evidence suggests
this to be the case.  Furthermore, when $N$ is odd, we know that $F_N$ must, in fact, have at least $\varphi(2N) + \varphi(N)$ roots on the unit circle. 
Again, our evidence suggests that there are no others.  Also, the identity
\begin{equation*}
	\varphi(2N) = \begin{cases}
		2\varphi(N) & \mbox{if } N \mbox{ is even} \\
		\varphi(N) & \mbox{if } N \mbox{ is odd}.
		\end{cases}
\end{equation*}
holds for all positive integers $N$.  So we pose the following strengthening of the Goldbach conjecture.
 
 \begin{conj} \label{UnitCircleRoots}
 	If $N > 5$ is an integer then $F_N$ has precisely $2\varphi(N)$ roots on the unit circle.
 \end{conj}

Similar to our note above, the converse of Conjecture \ref{UnitCircleRoots} is not necessarily true.  $F_N$ could have many roots on the unit circle
while still not being divisible by $\Phi_N$.  Once again, this conjecture should be regarded as more difficult than the Goldbach conjecture.

We also observe that Conjecture \ref{UnitCircleRoots} is a consequence of Conjecture \ref{GeneralGoldbach}.
Indeed, for the case of even $N$, if $F/\Phi_{2N}$ is irreducible and has a root on the unit circle, then it must be reciprocal, which it
certainly is not.  Similar remarks apply to $F/(\Phi_N\Phi_{2N})$ when $N$ is odd.

We have computed the number of roots of $F_N$ on the unit circle for $N\leq 50$ and have found that Conjecture \ref{UnitCircleRoots} holds for those $F_N$.
This complete list is given in Table \ref{fig:FRoots} including the number of roots inside, on and outside the unit circle for each $F_N$.

\begin{table}
\caption{Location of roots of $F_N$}\label{fig:FRoots}
\centering
\begin{tabular}{ c | c | c | c | c}
\hline\hline
$N$ & $2\varphi(N)$ & $[|z|<1 \quad |z| = 1\quad  |z| > 1]$ \\
\hline
$6$ & $4$ & $[16 \quad 4 \quad 30]$ \\
$7$ & $12$ & $[4\quad 12\quad 44]$ \\
$8$ &  $8$ & $[24\quad 8\quad 66]$ \\
$9$ & $12$ & $[8\quad 12\quad 92]$ \\
$10$ & $8$ & $[16\quad 8\quad 102]$ \\
$11$ & $20$ &  $[16\quad 20\quad 104]$\\
12 & 8& [48 \quad  8 \quad   186]\\
13 & 24 & [40\quad 24\quad 200]\\
 14& 12& [40\quad 12\quad 286]\\
15& 16& [40\quad 16\quad 308]\\
16& 16& [36\quad 16\quad 338]\\
17& 32& [36\quad 32\quad 348]\\
18 & 12& [56\quad 12\quad 510]\\
19& 36& [40\quad 36\quad 536]\\
20& 16& [80\quad 16\quad 626]\\
21& 24& [60\quad 24\quad 676]\\
22& 20& [64\quad 20\quad 714]\\
23& 44& [56\quad 44\quad 736]\\
24& 16& [92\quad 16\quad 950]\\
25& 40& [84\quad 40\quad 980]\\
26& 24& [100\quad 24\quad 1026]\\
27& 36& [108\quad 36\quad 1052]\\
28& 24& [92\quad 24\quad 1126]\\
29& 56& [100\quad 56\quad 1132]\\
30& 16& [132\quad 16\quad 1534]\\
 31& 60& [128\quad 60\quad 1552]\\
32& 32& [144\quad 32\quad 1746]\\
33& 40& [136\quad 40\quad 1808]\\
34& 32& [144\quad 32\quad 1870]\\
35& 48& [160\quad 48\quad 1900]\\
36& 24& [168\quad 24\quad 1978]\\
37& 72& [136\quad 72\quad 2024]\\
38& 36& [180\quad 36\quad 2522]\\
39& 48& [172\quad 48\quad 2592]\\
40& 32& [184\quad 32\quad 2670]\\
 41& 80& [176\quad 80\quad 2704]\\
42& 24& [200\quad 24\quad 3138]\\
43& 84& [184\quad 84\quad 3176]\\
44& 40& [244\quad 40\quad 3414]\\
 45& 48& [252\quad 48\quad 3484]\\
46& 44& [228\quad 44\quad 3598]\\
47& 92& [244\quad 92\quad 3620]\\
48& 32& [288\quad 32\quad 4098]\\
49& 84& [260\quad 84\quad 4168]\\
50& 40& [264\quad 40\quad 4302]\\

\end{tabular}
\end{table}

It is worth noting that, in our construction of $F_N$, the set of odd primes may be replaced with any subset of $\nat$.  In this way, one may
attempt to prove theorems analogous to those stated above.  One such example, which is of particular interest in number theory, arises in the following way.

The Liouville function $\lambda:\nat\to\{-1,1\}$ is the completely multiplicative function such that $\lambda(p) = -1$ at every prime $p$.
Now define the set
\begin{equation*}
  \mathcal L = \{n\in \nat: \lambda(n) = -1\}.
\end{equation*}
It is a direction of our future research to examine the analogs of $F_N$ that are obtained by using the above construction with $\mathcal L$ in place of $\PP$.
Perhaps this strategy can yield a proof that every positive even integer $N>2$ satisfies $N\in\mathcal L + \mathcal L$.  On the surface, such a result appears to 
be easier than the Goldbach conjecture, and therefore, is possibly within reach.

One can also consider weighted forms of $F_N$.  Similar to the study of the prime number theorem, instead of using the above indicator
function of $\PP$, we use the weighted form
\begin{equation*}
  \overset{\sim}{\chi}_{\PP}(n) =\begin{cases}  \log n & \mbox{ if $n \in \PP $,} \\ 0 & \mbox{ otherwise} \end{cases}
\end{equation*}
and define the corresponding polynomials $\overset{\sim}{F}_N$ by
\begin{equation*}
  \overset{\sim}{F}_N(z) =\sum_{k=0}^{N-1}\left(\sum_{n=1}^{N-1}\overset{\sim}{\chi}_{\PP}(n)z^{kn}\right)^2.
\end{equation*}
It is clear that the $\overset{\sim}{F}_N(z)$ do not have integer coefficients, so we might expect different types of results regarding these polynomials.
Nonetheless, we believe they yield another interesting route for future research.

In the following two sections, we examine a series of basic properties of the polynomials $F_N$.  Specifically in section \ref{Coefficients}, we produce estimates on the size of the
coefficients of $F_N$, as well as asymptotic formulae for certain sums of their coefficients.  The remaining sections are devoted to the proofs of our results.
 
\section{Properties of the polynomials $F_N$} \label{Properties}

Now that we understand the relevance of the polynomials $F_N$ to the Goldbach conjecture, we consider some of their additional properties.    
We begin with the following result regarding their symmetry.

\begin{thm} \label{Symmetry}
  If $N$ is a positive integer then $F_N(z) = F_N(-z)$.
\end{thm}

Theorem \ref{Symmetry} certainly implies that if $\Phi_N(z)$ divides $F_N(z)$ then so does $\Phi_N(-z)$.  Furthermore, we know that if $M$ is an odd integer
then $\Phi_{2M}(z) = \Phi_M(-z)$.  Combining these observations with Theorem \ref{FDivisibility}, we obtain the following corollary.

\begin{cor} \label{GoldbachEquivalencyGeneral}
  If $M$ is an odd integer and $N = 2M$ then the following conditions are equivalent.
  \begin{enumerate}[(i)]
  \item\label{NDiv} $\Phi_N$ divides $F_N$.
  \item\label{MDiv} $\Phi_M$ divides $F_N$.
  \item\label{GoldbachDiv} The Goldbach conjecture fails for $N$.
  \end{enumerate}
\end{cor}

Suppose now that, for any positive integer $M$, $\zeta_M$ is a primitive $M$th root of unity.
We may view Corollary \ref{GoldbachEquivalencyGeneral} as examining the value of $F_N(\zeta_M)$ when $M$ is a certain divisor of $N$.  Next, we consider
the values of $F_N(\zeta_M)$ when $M$ is an arbitrary divisor of $M$.  We write $[x]$ to denote the largest integer less than or equal to $x$.

\begin{thm} \label{PropertyList}
  If $N>4$ is an integer and $M\mid N$ then the following conditions hold.
  \begin{enumerate}[(i)]
    \item\label{OddDivisor} If $M$ is odd then
      \begin{equation*}
        F_N(\zeta_M) \geq N \sum_{n=1}^{[N/2M]} R(2nM).
      \end{equation*}
    \item \label{EvenDivisor} If $M$ is even then
      \begin{equation*}
        F_N(\zeta_M) \geq N\sum_{n=1}^{N/M} R(nM).
      \end{equation*}
\end{enumerate}
\end{thm}

Applying Theorems \ref{PropertyList} and \ref{FDivisibility} immediately yield the following simpler lower bound on $F_N(\zeta_M)$.

\begin{cor} \label{SimpleZetaLower}
  If $N>4$ is an integer and $M\mid N$, then $F_N(\zeta_M) \geq NR(N)$ with equality when $M=N$.
\end{cor}

The case $M=N$ may not be the only case of equality in Corollary \ref{SimpleZetaLower}.  In fact, if $M$ is odd and $N = 2M$, then it can be shown 
that $F_N(\zeta_M) = NR(N)$ as well.  This result also provides a strengthening of one direction of Theorem \ref{FDivisibility}.
If $\Phi_M$ ever divides $F_N$, then it follows from Corollary \ref{SimpleZetaLower} that $R(N) = 0$.  In other words, we have established the following statement.

\begin{cor} \label{StrongerDirection}
   Suppose $N>4$ is an integer and $M\mid N$.  If $\Phi_M$ divides $F_N$ then the Goldbach conjecture fails for $N$.
\end{cor}

The converse of Corollary \eqref{StrongerDirection} is certainly false.  Otherwise, $\Phi_1$ would divide $F_N$ for every odd $N$, and it certainly 
does not.  When restricted to even integers, it is likely true, but only because the Goldbach conjecture would imply that the hypothesis is always false.
In fact, in view of Theorem \ref{FDivisibility}, such a statement is equivalent to the Goldbach conjecture.

\section{The coefficients of $F_N$} \label{Coefficients}

 Let us now turn our attention to understanding the coefficients of $F_N$.  For this purpose, we note that $\deg F_N \leq 2(N-1)^2$ and write
\begin{equation*}
      F_N(z) = \sum_{m = 0}^{2(N-1)^2}a_{N,m}z^m.
\end{equation*}
It is easy to see that the constant term in $F_N$ is given by the formula
\begin{equation*}
a_{N,0} = \left(\sum_{n=1}^{N-1}\chi_\PP(n)\right)^2 = (\pi(N-1)-1)^2
\end{equation*}
where $\pi(N-1)$ denotes the number of primes $p\leq N-1$.   Furthermore, by multiplying out the terms in the definition of $F_N$,
we obtain an explicit formula for all other coefficients of $F_N$.

\begin{thm} \label{Fcoeffs}
      Let $N$ be a positive integer.  We have that
      \begin{equation*}  \label{ExplicitA}
          a_{N,m}  =  \sum_{\substack{d\mid m \\ m/d < N}}\sum_{n=\max\{0,d-N\}+1}^{\min\{N,d\}-1}\chi_\PP(n)\chi_\PP(d-n)
	\end{equation*}
      for all $0<m\leq 2(N-1)^2$.
\end{thm}

Among other things, Theorem \ref{Fcoeffs} shows that 
\begin{equation*}
      a_{N,m} \leq \sum_{d\mid m} R(d)
\end{equation*}
with equality whenever $0< m \leq N$.  We can rephrase the case of equality by saying that
\begin{equation} \label{CoeffIdentity}
	a_{N,m} = \sum_{d\mid m} R(d)
\end{equation}
whenever $0<m \leq N$.  It is worth noting that the right hand side of \eqref{CoeffIdentity} does not depend on $N$, so that the non-constant coefficients
of the $F_N(z)$ stabilize as $N$ tends to infinity.  More specifically, if we write $a(m) = a_{N,m}$ for some $N\geq m$, the polynomials $F_N(z) - a_{N,0}$
converge coefficient-wise to the power series
\begin{equation*}
	F(z) = \sum_{n=1}^\infty a(m)z^m.
\end{equation*}
It is straightforward to verify that $F(z)$ has radius of convergence $1$, and the sequence $\{F_N(z) - a_{N,0}\}$ converges uniformly to $F(z)$
on compact subsets of the unit disk.

Let us now examine the individual terms $a(m)$.  If $m$ is odd, then all divisors of $m$ are also
odd, so we conclude that $a(m)=0$.  Hence, it is only interesting to consider the situation where $m$ is even, in which case the
coefficients seem to behave in a rather subtle way.  However, we can obtain lower bounds
in relation to other famous arithmetic functions.  Before proceeding, we recall that $\omega(n)$ denotes the number of distinct prime
factors of $n$ and $d(n)$ denotes the number of divisors of $n$.

\begin{thm} \label{aLowerBounds}
      If $m>1$ is an integer then
      \begin{equation} \label{UnconditionalLower}
          a(2m) \geq \omega(m) -  \left\{\begin{array}{ll}
                  1 & \mathrm{if}\ m\equiv 2\ \mathrm{mod}\ 4, \\
                  0 & \mathrm{otherwise}.
                  \end{array}\right.
      \end{equation}
      Moreover, if the Goldbach conjecture is true, then
      \begin{equation} \label{UnderGoldbach}
          a(2m) \geq d(m) - \left\{\begin{array}{ll}
                  2 & \mathrm{if}\ m\ \mathrm{is\ even}, \\
                  1 & \mathrm{otherwise}.
                  \end{array}\right.
      \end{equation}
\end{thm}

We note that the right hand side of \eqref{UnconditionalLower} is always positive for $m>2$.  So taking an integer $m>4$, we have that
$a(m) = 0$ if and only if $m$ is odd.  It is also worth observing that the right hand sides of \eqref{UnconditionalLower} and \eqref{UnderGoldbach}
are sometimes equal, namely when $m$ is prime.  In general, however, $d(m)$ is much larger than $\omega(m)$ so that our bound under the Goldbach conjecture is
stronger than the analogous unconditional bound.

It is reasonable to expect that, not only is $R(2d)$ positive for $d>2$, but it is quite large most of the time.  More specifically, Hardy and Littlewood have
proposed the following asymptotic formula.

\begin{conj}[Hardy and Littlewood~\cite{HL}] \label{HLConj}
As $n$ tends to infinity,
\begin{equation}
  R(2n) \sim 2C_2 \frac{n}{\log^2n} \prod_{\substack{p|n \\ p > 2}}\frac{p-1}{p-2},
\end{equation}
where $C_2$ is the twin primes constant
\begin{equation*}
	C_2 = \prod_{p>2}\left( 1-\frac{1}{(p-1)^2}\right).
\end{equation*}
\end{conj}

Under the assumption of Conjecture \ref{HLConj}, we can improve the bounds of Theorem \ref{aLowerBounds}.  If $2^k\parallel m$, then define
\begin{equation} \label{JDef}
	J(m) = \bigg( 2-\frac1{2^k} \bigg) \prod_{\substack{p^\ell\parallel m \\ p>2}} \bigg( 1-\frac2{p^{\ell+1}} \bigg) \bigg( 1-\frac2p \bigg)^{-1}.
\end{equation}
Here, $p^\ell \parallel m$ means that $p^\ell\mid m$ but $p^{\ell+1}\nmid m$.

\begin{thm} \label{HLConjResult}
If Conjecture \ref{HLConj} is true, then 
\begin{equation*}
	a(2m) \sim \frac{2C_2 J(m) m}{\log^2m}
\end{equation*}
as $m$ tends to infinity.
\end{thm}

For a positive integer $M$, it is also of interest to study the summatory function
\begin{equation*}
      A(M) = \sum_{m=1}^{2M} a(m).
\end{equation*}
By applying Theorem \ref{aLowerBounds} directly, we are able to verify that
\begin{equation*}
	A(M) \geq \sum_{m=1}^M \omega(m)  + O(M) = M\log\log M + O(M),
\end{equation*}
where the last equality is obtained from \cite{HardyWright}, page 355.  If we are willing to assume the Goldbach conjecture,
a similar argument reveals that
\begin{equation} \label{GLower}
	A(M) \geq \sum_{m=1}^M d(m) + O(M) = M\log M + O(M).
\end{equation}
As we have remarked following our statement of Theorem \ref{aLowerBounds}, we anticipate that $a(2m)$ is large much of the time.
However, in order to obtain an asymptotic formula for $a(2m)$, we needed to assume a very strong conjecture of Hardy and Littlewood.
In the case of $A(M)$, we can obtain such a formula unconditionally.

\begin{thm} \label{UnconditionalABound}
  We have that
  \begin{equation*}
    A(M) = \frac{\pi^2M^2}{3\log^2M} + O\left( \frac{M^2\log\log M}{\log^3M} \right).
  \end{equation*}
\end{thm}

\section{Proofs of the results from section \ref{PolyIntro}}

We begin this section with the proof to Theorem \ref{FDivisibility}

\begin{proof}[Proof of Theorem \ref{FDivisibility}]
Let $\zeta$ be a primitive $N$th root of unity.  We have immediately that
\begin{align*}
  F_N(\zeta) & = \sum_{k=0}^{N-1}\left(\sum_{n=1}^{N-1}\chi_\PP(n)\zeta^{kn}\right)^2 \\
  & = \sum_{k=0}^{N-1}\sum_{m=1}^{N-1}\sum_{n=1}^{N-1}\chi_\PP(m)\chi_\PP(n)\zeta^{k(m+n)} \\
  & =  \sum_{m=1}^{N-1}\sum_{n=1}^{N-1}\chi_\PP(m)\chi_\PP(n)\sum_{k=0}^{N-1}\zeta^{k(m+n)}.
      \end{align*}
      We know that
      \begin{equation*}
          \sum_{k=0}^{N-1}\zeta^{k(m+n)} = 0
      \end{equation*}
unless $m+n\equiv 0\mod N$.  In our case, this may occur only
when $m+n= N$, implying that
      \begin{equation*}
          F_N(\zeta) = \sum_{n=1}^{N-1} \chi_\PP(n)\chi_\PP(N-n)\sum_{k=0}^{N-1}\zeta^{kN} = NR(N).
\end{equation*}
If $R(N) = 0$ then $F_N(\zeta) = 0$ showing that $\Phi_N$ must divide $F_N$.  On the other hand, if $\Phi_N$ divides $F_N$, it is
obvious that $F_N(\zeta) = 0$ so that $R(N) = 0$.
\end{proof}

We already have all of the tools necessary to prove Theorem \ref{2NCyclo}.

\begin{proof}[Proof of Theorem \ref{2NCyclo}]
We must show that $F_N(e^{\pi i/N}) = 0$.  To see this, note that
\begin{align*}
	F_N(e^{\pi i/N}) & = \sum_{k=0}^{N-1} \left(\sum_{n=1}^{N-1}\chi_\PP(n)e^{\frac{\pi ikn}{N}}\right)^2 \\
              & = \sum_{k=0}^{N-1} \sum_{m=1}^{N-1}\sum_{n=1}^{N-1}\chi_\PP(m)\chi_\PP(n)e^{\frac{\pi ik(m+n)}{N}} \\
		& =  \sum_{m=1}^{N-1}\sum_{n=1}^{N-1}\chi_\PP(m)\chi_\PP(n)\sum_{k=0}^{N-1}e^{\frac{\pi ik(m+n)}{N}}.
\end{align*}
The product $\chi_\PP(m)\chi_\PP(n) = 0$ unless $m$ and $n$ are both odd primes.  In this case, we certainly have that $m+n$ is even so that
\begin{equation} \label{EvenTransform}
	\sum_{k=0}^{N-1}e^{\frac{\pi ik(m+n)}{N}} = \sum_{k=0}^{N-1}e^{\frac{2\pi ik((m+n)/2)}{N}}.
\end{equation}
Of course, $0 < (m+n)/2 < N$ implying that the right hand side of \eqref{EvenTransform} equals zero.  In other words, we have shown that
\begin{equation*}
	\chi_\PP(m)\chi_\PP(n)\sum_{k=0}^{N-1}e^{\frac{\pi ik(m+n)}{N}} = 0
\end{equation*}
for all $1\leq m,n < N$, verifying the theorem.
\end{proof}

\section{Proofs of the results from section \ref{Properties}}

\begin{proof}[Proof of Theorem \ref{Symmetry}]
  It follows directly from the definition that
  \begin{equation} \label{MinusDefinition}
    F_N(-z) = \sum_{k=0}^{N-1}\left(\sum_{n=1}^{N-1}(-1)^{kn}\chi_\PP(n)z^{kn}\right)^2.
  \end{equation}
  If $n$ is even, we certainly have that $\chi_\PP(n) = 0$.  Otherwise, we have that $(-1)^n = -1$, which implies that
  $(-1)^{kn}\chi_\PP(n) = (-1)^{k}\chi_\PP(n)$ for all $n$.  Using \eqref{MinusDefinition}, we find that
  \begin{equation*}
    F_N(-z) = \sum_{k=0}^{N-1}\left((-1)^k\sum_{n=1}^{N-1}\chi_\PP(n)z^{kn}\right)^2 = \sum_{k=0}^{N-1}\left(\sum_{n=1}^{N-1}\chi_\PP(n)z^{kn}\right)^2 = F_N(z)
  \end{equation*}
  which completes the proof.
\end{proof}
   
In view of Theorem \ref{Symmetry}, we obtain our proof of Corollary \ref{GoldbachEquivalencyGeneral} almost immediately.

\begin{proof}[Proof of Corollary \ref{GoldbachEquivalencyGeneral}]
  In view of Theorem \ref{FDivisibility}, we immediately have that \eqref{NDiv} if and only if \eqref{GoldbachDiv}.  To finish the proof, we will show that
  \eqref{NDiv} if and only if \eqref{MDiv}.  To see this, note that since $M$ is odd, we have that $\Phi_N(z) = \Phi_M(-z)$.  Furthermore, Theorem \ref{Symmetry}
  implies that $\Phi_N(z)$ divides $F_N(z)$ if and only if $\Phi_N(-z)$ divides $F_N$ and the result follows.
\end{proof}

\begin{proof}[Proof of Theorem \ref{PropertyList}]
  Suppose that $a=1$ if $M$ is odd and $a=0$ if $M$ is even.  We must show that
  \begin{equation*} \label{ComplicatedVersion}
    F_N(\zeta_M) \ge N \sum_{1\le k \le N/(2^aM)} R(2^akM).
  \end{equation*}
  From the definition of $F_N$, we have that
  \begin{align*}
       F_N(\zeta_M) & = \sum_{k=0}^{N-1} \sum_{2 < p_1,p_2 \le N-1}\zeta_M^{k(p_1+p_2)} \\
       & = \sum_{2 < p_1,p_2 \le N-1} \sum_{i=0}^{N/M-1}\sum_{k=0}^{M-1}\zeta_M^{(iM+k)(p_1+p_2)} \\
       & = \frac{N}{M}\sum_{2 < p_1,p_2 \le N-1}\sum_{k=0}^{M-1}\zeta_M^{k(p_1+p_2)}.
\end{align*}
Now the inner summation over $k$ is zero unless $(p_1+p_2)/M \in\intg$. Hence we have
\begin{align*}
       F_N(\zeta_M) & = N\sum_{1\le \ell \le 2(N-1)/M} \sum_{\substack{2 < p_1,p_2 \le N-1 \\ p_1+p_2=\ell M}}1 \\
       & = N\left\{\sum_{1\le \ell \le N/M}+\sum_{N/M+1\le \ell \le 2(N-1)/M}\right\}\sum_{\substack{2 < p_1,p_2 \le N-1 \\ p_1+p_2=\ell M}}1 \\
       & = N\sum_{1\le \ell \le N/(2^aM)} R(2^a\ell M) + N \sum_{N/M+1\le \ell  \le 2(N-1)/M}\sum_{\substack{2 < p_1,p_2 \le N-1 \\ p_1+p_2=\ell M}}1 \\
       & \ge N \sum_{1\le \ell  \le N/(2^aM)} R(2^a\ell M).
\end{align*}
and the result follows.
\end{proof}

\begin{proof}[Proof of Corollary \ref{SimpleZetaLower}]
  If $M$ is even, we have that
  \begin{equation*}
    F_N(\zeta_M) \geq N \sum_{n=1}^{N/M} R(nM) \geq N R\left(\frac{N}{M}\cdot M\right) = NR(N).
  \end{equation*}
  If $M$ is odd and $N$ is even, then $N/2M\in\nat$ so it follows that
  \begin{equation*}
    F_N(\zeta_M) \geq N\sum_{n=1}^{N/2M} R(2nM) \geq N R\left(2\cdot\frac{N}{2M}\cdot M\right) = NR(N).
  \end{equation*}
  Finally, if $M$ and $N$ are both odd, then $NR(N) = 0$ so that
  \begin{equation*}
    F_N(\zeta_M) \geq N\sum_{n=1}^{[N/2M]} R(2nM) \geq 0 = NR(N).
  \end{equation*}
\end{proof}

\begin{proof}[Proof of Corollary \ref{StrongerDirection}]
  If $\Phi_M\mid F_N$ then we have that $F_N(\zeta_M) = 0$.  It follows from Corollary \ref{SimpleZetaLower} that $R(N) = 0$.
\end{proof}

\section{Proofs of the results from section \ref{Coefficients}}

\begin{proof}[Proof of Theorem \ref{Fcoeffs}]
      We first note that
      \begin{equation*}
         F_N(z) = \sum_{k=0}^{N-1}\left(\sum_{n=1}^{N-1}\chi_\PP(n)z^{kn}\right)^2 = \sum_{k=0}^{N-1}\sum_{m=1}^{N-1}\sum_{n=1}^{N-1}\chi_\PP(m)\chi_\PP(n)z^{k(m+n)}.
       \end{equation*}
       Relabeling the indices, we find that
       \begin{align*}
	F_N(z)	& = \sum_{m = 0}^{2(N-1)^2}\left( \sum_{\substack{d\mid m \\ m/d < N}}\sum_{\substack{n_1+n_2 = d \\ 1\leq n_1,n_2< N}}\chi_\PP(n_1)\chi_\PP(n_2)\right)z^m \\
		& = \sum_{m = 0}^{2(N-1)^2}\left( \sum_{\substack{d\mid m \\ m/d < N}}\sum_{n=\max\{0,d-N\}+1}^{\min\{N,d\}-1}\chi_\PP(n)\chi_\PP(d-n)\right)z^m
      \end{align*}
      establishing the theorem.
\end{proof}

\begin{proof}[Proof of Theorem \ref{aLowerBounds}]
      Using Theorem \ref{CoeffIdentity}, we have immediately that
      \begin{equation*}
a(2m) = \sum_{d\mid 2m} \sum_{n=1}^{d-1}
\chi_\PP(n)\chi_\PP(d-n).
      \end{equation*}
      However, it is clear that
      \begin{equation*}
          \sum_{n=1}^{d-1} \chi_\PP(n)\chi_\PP(d-n) = 0
      \end{equation*}
      whenever $d$ is odd, which implies that
      \begin{eqnarray}
a(2m) & = & \sum_{\substack{d\mid 2m \\ d\ \mathrm{even}}}
\sum_{n=1}^{d-1} \chi_\PP(n)\chi_\PP(d-n) \nonumber \\
& = & \sum_{d\mid m} \sum_{n=1}^{2d-1}
\chi_\PP(n)\chi_\PP(2d-n). \label{EvenIdentity}
      \end{eqnarray}

We now use \eqref{EvenIdentity} to prove
\eqref{UnconditionalLower}.  If $p$ is an odd prime, we have that
$\chi_\PP(p)\chi_\PP(2p-p) = 1$ implying
      \begin{equation} \label{OddPrimeLower}
          \sum_{n=1}^{2p-1} \chi_\PP(n)\chi_\PP(2p-n) \geq 1.
      \end{equation}
Now let $\omega_{\mathrm{odd}}(m)$ denote the number of distinct
odd prime divisors of $m$ and consider three cases according to the
residue class
      of $m$ modulo $4$.
      \begin{enumerate}[(i)]

\item First assume that $m$ is odd.  In this case, we have that
$\omega_{\mathrm{odd}}(m) = \omega(m)$ and $m\not\equiv 2\
\mathrm{mod}\ 4$.
The inequality \eqref{OddPrimeLower} holds for every odd prime
divisor or $m$.  Combining this observation with
\eqref{EvenIdentity}, we find that
      \begin{equation*}
          a(2m) \geq \omega_{\mathrm{odd}}(m) = \omega(m)
      \end{equation*}
      completing the proof in this case.

\item Now assume that $m \equiv 0\ \mathrm{mod}\ 4$.  It is
easily verified that
      \begin{equation*}
          \sum_{d\mid 4} \sum_{n=1}^{7} \chi_\PP(n)\chi_\PP(8-n) = 1,
      \end{equation*}
and then it follows from \eqref{EvenIdentity} and
\eqref{OddPrimeLower} that
      \begin{equation*}
          a(2m) \geq \omega_{\mathrm {odd}}(m) + 1.
      \end{equation*}
Since $2$ divides $m$, we have that  $\omega_{\mathrm {odd}}(m) = \omega(m) -1$ establishing the result in this case.

\item Finally, we consider the case that $m\equiv 2\
\mathrm{mod}\ 4$.  Again, $m$ is even so that $\omega_{\mathrm{odd}}(m) = \omega(m) -1$,
and we conclude from \eqref{EvenIdentity} and
\eqref{OddPrimeLower} that $a(2m) \geq \omega_{\mathrm {odd}}(m)$.
This completes the proof of \eqref{UnconditionalLower}.
\end{enumerate}

To establish \eqref{UnderGoldbach}, we assume that the Goldbach
Conjecture holds.  Hence, we have that
      \begin{equation} \label{GoldbachLowerBound}
          \sum_{n=1}^{2d-1} \chi_\PP(n)\chi_\PP(2d-n) \geq 1
      \end{equation}
for all divisors $d$ of $m$ with $d\not\in\{1,2\}$.  Here we consider
two cases.
      \begin{enumerate}[(i)]

\item Suppose first that $m$ is odd.  Here, we have that
\eqref{GoldbachLowerBound} holds for all divisors $d$ of $m$
different than $1$.  This gives
\begin{align*}
	a(2m) = \sum_{d\mid m} \sum_{n=1}^{2d-1}\chi_\PP(n)\chi_\PP(2d-n) & =  \sum_{\substack{d\mid m \\ d\ne 1}}\sum_{n=1}^{2d-1} \chi_\PP(n)\chi_\PP(2d-n) \\
	 & \geq  \sum_{\substack{d\mid m \\ d\ne 1}} 1  = d(m) - 1
      \end{align*}
      completing the proof in this case.

\item In the case that $m$ is even, we have that
\eqref{GoldbachLowerBound} holds except when $d=1$ or $d=2$.
Therefore, we have that
 \begin{align*}
	a(2m) = \sum_{d\mid m} \sum_{n=1}^{2d-1}\chi_\PP(n)\chi_\PP(2d-n) & = \sum_{\substack{d\mid m \\ d\not\in\{1,2\}}}\sum_{n=1}^{2d-1} \chi_\PP(n)\chi_\PP(2d-n) \\
	& \geq \sum_{\substack{d\mid m \\ d\not\in\{1,2\}}} 1 = d(m) - 2
      \end{align*}
      which completes the proof in this case as well.
      \end{enumerate} 
\end{proof}

We now move on to a proposition from which we will deduce Theorem~\ref{HLConjResult}. Define
\[
f(n) = \prod_{\substack{p|n \\ p > 2}}\frac{p-1}{p-2}
\]
to be the multiplicative function appearing in Conjecture~\ref{HLConj}, and note that if $k\ge0$ is the integer such that $2^k\parallel m$, then
\begin{align}
\sum_{d\mid m} df(d) &= \prod_{p^\ell\parallel m} \sum_{d\mid p^\ell} df(d) \notag \\
&= \prod_{p^\ell\parallel m} \big( 1 + pf(p) + p^2f(p^2) + \cdots + p^\ell f(p^\ell) \big) \notag \\
&= \bigg( 1 + 2 \frac{2^k-1}{2-1} \bigg) \prod_{\substack{p^\ell\parallel m \\ p>2}} \bigg( 1 + \frac{p-1}{p-2} \cdot p \frac{p^\ell-1}{p-1} \bigg) \notag \\
&= (2^{k+1}-1) \prod_{\substack{p^\ell\parallel m \\ p>2}} \frac{p^{\ell+1}-2}{p-2} = mJ(m)
\label{df(d) sum}
\end{align}
by comparison with \eqref{JDef}.

\begin{prop}
\label{using epsilons prop}
Let $0<\ep\le\frac12$ be given. Suppose there exists a positive integer $n(\ep)$ such that
\begin{equation}
\label{HL with epsilons}
(1-\ep)2C_2 f(n) \frac n{\log^2n} \le R(2n) \le (1+\ep)2C_2 f(n) \frac n{\log^2n}
\end{equation}
for all $n>n(\ep)$. Then there exists a constant $m(\ep)$ such that
\begin{equation}
\label{our theorem with epsilons}
(1-2\ep)2C_2 J(m) \frac m{\log^2m} \le a(2m) \le (1+11\ep)2C_2 J(m) \frac m{\log^2m}
\end{equation}
for all $m>m(\ep)$.
\end{prop}

\noindent It is clear that Theorem~\ref{HLConjResult} follows from Proposition~\ref{using epsilons prop}, since Conjecture~\ref{HLConj} implies that the hypothesis of Proposition~\ref{using epsilons prop} holds for every $\ep>0$.

\begin{proof}[Proof of Proposition~\ref{using epsilons prop}]
We shall not keep track explicitly of the necessary value for $m(\ep)$, instead simply saying ``when $m$ is large enough'' (in terms of $\ep$) in the appropriate places. We begin by writing
\begin{equation}
\label{split at 1-epsilon}
a(2m) = \sum_{c\mid 2m} R(c) = \sum_{d\mid m} R(2d) = \sum_{\substack{d\mid m \\ d\le m^{1-\ep}}} R(2d) + \sum_ {\substack{d\mid m \\ d>m^{1-\ep}}} R(2d)
\end{equation}
(where the second equality uses the fact that $R(c)=0$ when $c$ is odd).

First we establish the upper bound in equation~\eqref{our theorem with epsilons}. We have $m^{1-\ep} > n(\ep)$ when $m$ is large enough, and so the summands in the second sum on the right-hand side of equation~\eqref{split at 1-epsilon} can be bounded above by the upper bound in equation~\eqref{HL with epsilons}. For the first sum on the right-hand side we simply use the trivial bound $R(2n) \le n$. The result is
\begin{align*}
a(2m) &\le \sum_{\substack{d\mid m \\ d\le m^{1-\ep}}} d + \sum_ {\substack{d\mid m \\ d>m^{1-\ep}}} (1+\ep)2C_2 f(d) \frac{d}{\log^2d} \\
&\le \sum_{\substack{d\mid m \\ d\le m^{1-\ep}}} m^{1-\ep} + (1+\ep)2C_2 \frac1{(1-\ep)^2\log^2m} \sum_ {\substack{d\mid m \\ d>m^{1-\ep}}} df(d) \\
&= m^{1-\ep}\tau(m) + \frac{1+\ep}{(1-\ep)^2} \frac{2C_2}{\log^2m} mJ(m)
\end{align*}
using the identity~\eqref{df(d) sum}, where $\tau(n)$ denotes the number of divisors of~$n$. It is well known that $\tau(m) \ll_\ep m^{\ep/3}$, and so the first term is less than $\ep m/\log^2m$ when $m$ is large enough. Also $(1+\ep)/(1-\ep)^2 \le 1+10\ep$ for $0<\ep\le\frac12$. Therefore
\begin{align*}
a(2m) &\le \ep \frac m{\log^2m} + (1+10\ep) \frac{2C_2}{\log^2m} mJ(m) \le (1+11\ep) 2C_2 J(m) \frac m{\log^2m}
\end{align*}
when $m$ is large enough, since $J(m)\ge1$ for all positive integers $m$ and $2C_2>1$. This establishes the upper bound in equation~\eqref{our theorem with epsilons}.

A similar method addresses the lower bound in equation~\eqref{our theorem with epsilons}. Since $m^{1-\ep} > n(\ep)$ when $m$ is large enough, the summands in the second sum on the right-hand side of equation~\eqref{split at 1-epsilon} can be bounded below by the lower bound in equation~\eqref{HL with epsilons}; the first sum on the right-hand side is nonnegative, and so we can simply delete it. We obtain the lower bound
\begin{align}
a(2m) &\ge \sum_ {\substack{d\mid m \\ d>m^{1-\ep}}} (1+\ep)2C_2 f(d) \frac{d}{\log^2d} \notag \\
&\ge (1-\ep)\frac{2C_2}{\log^2m} \sum_ {\substack{d\mid m \\ d>m^{1-\ep}}} df(d) = (1-\ep)\frac{2C_2}{\log^2m} \bigg( mJ(m) - \sum_ {\substack{d\mid m \\ d\le m^{1-\ep}}} df(d) \bigg),
\label{taking care of difference}
\end{align}
again using the identity~\eqref{df(d) sum}. This last sum is bounded above by
\begin{equation*}
\sum_ {\substack{d\mid m \\ d\le m^{1-\ep}}} df(d) \le \sum_{d\mid m} \bigg( \frac{m^{1-\ep}}d \bigg)^{1+\ep/2} df(d) \le m^{1-\ep/2} \sum_{d\mid m} \prod_{\substack{p|d \\ p > 2}}\frac{p-1}{p^{\ep/2}(p-2)}.
\end{equation*}
There are only finitely many primes $p$ for which $(p-1)/p^{\ep/2}(p-2)$ exceeds 1, and so the inner product on the right-hand side is uniformly bounded by some constant $C(\ep)$. Therefore
\[
\sum_ {\substack{d\mid m \\ d\le m^{1-\ep}}} df(d) \le C(\ep) m^{1-\ep/2} \sum_{d\mid m} 1 = C(\ep) m^{1-\ep/2} \tau(m),
\]
which as above is less than $\ep m$ for $m$ large enough. Therefore equation~\eqref{taking care of difference} becomes
\[
a(m) \ge (1-\ep)\frac{2C_2}{\log^2m}( mJ(m) - \ep m) \ge (1-2\ep) 2C_2J(m) \frac m{\log^2m}
\]
when $m$ is large enough, again since $J(m)\ge1$ always. This establishes the lower bound in equation~\eqref{our theorem with epsilons}.
\end{proof}

Before we begin the proof of Theorem \ref{UnconditionalABound}, we will require a lemma regarding the function
\begin{equation*}
Q(x) = \sum_{p+q\le x} 1,
\end{equation*}
where $p$ and $q$ denote primes.

\begin{lem} \label{Q lemma}
  Uniformly for $x\ge3$,
  \begin{equation*}
    Q(x) = \frac{x^2}{2\log^2 x} + O\left( \frac{x^2\log\log x}{\log^3 x} \right).
  \end{equation*}
\end{lem}

\begin{proof}
We begin by writing
\[
Q(x) = \sum_{p\le x} \pi(x-p) = \sum_{x/\log x \le p \le x-\sqrt x} \pi(x-p) + O\bigg( \sum_{p\le x/\log x} \pi(x-p) + \sum_{x-\sqrt x\le p\le x} \pi(x-p) \bigg).
\]
Trivially $\pi(x-p) \le \pi(x) \le x$, so
\begin{align}
  Q(x) &= \sum_{x/\log x \le p \le x-\sqrt x} \pi(x-p) + O\bigg( \sum_{p\le x/\log x} \pi(x) + \sum_{x-\sqrt x\le p\le x} x \bigg) \notag \\
  &= \sum_{x/\log x \le p \le x-\sqrt x} \pi(x-p) + O\bigg( \pi(x) \pi\bigg( \frac x{\log x} \bigg) + x\sqrt x \bigg) \notag \\
  &= \sum_{x/\log x \le p \le x-\sqrt x} \pi(x-p) + O\bigg( \frac{x^2}{\log^3x} \bigg).
  \label{deal with main term}
\end{align}
In the main term, the prime number theorem gives
\[
\sum_{x/\log x\le p\le x-\sqrt x} \pi(x-p) = \sum_{x/\log x \le p \le x-\sqrt x} \bigg( \li(x-p) + O\bigg( \frac{x-p}{\log^2(x-p)} \bigg) \bigg)
\]
(we could insert a better error term, but it would not improve the final result). Since $x-p \ge \sqrt x$, we have $\log(x-p) \gg \log x$ and so
\begin{align*}
&= \sum_{x/\log x \le p \le x-\sqrt x} \li(x-p) + O\bigg( \sum_{x/\log x \le p \le x-\sqrt x} \frac x{\log^2x} \bigg) \\
&= \sum_{x/\log x \le p \le x-\sqrt x} \li(x-p) + O\bigg( \frac x{\log^2x} \pi(x) \bigg) \\
&= \sum_{x/\log x \le p \le x-\sqrt x} \li(x-p) + O\bigg( \frac{x^2}{\log^3x} \bigg),
\end{align*}
which transforms equation~\eqref{deal with main term} into
\begin{equation}
Q(x) = \sum_{x/\log x \le p \le x-\sqrt x} \li(x-p) + O\bigg( \frac{x^2}{\log^3x} \bigg).
  \label{li sum}
\end{equation}

Using partial summation, we have
\begin{align*}
  \sum_{x/\log x \le p \le x-\sqrt x} & \li(x-p) = \int_{x/\log x}^{x-\sqrt x} \li(x-t) \,d\pi(t) \\
  &= \pi(x-\sqrt x)\li(\sqrt x) - \pi\bigg( \frac x{\log x} \bigg) \li\bigg( x - \frac x{\log x} \bigg) + \int_{x/\log x}^{x-\sqrt x} \frac{\pi(t)}{\log(x-t)} \,dt,
\end{align*}
since the $t$-derivative of $\li(x-t)$ is $-1/\log(x-t)$. In other words,
\begin{align*}
  \sum_{x/\log x \le p \le x-\sqrt x} \li(x-p) &= O\bigg( x\sqrt x + \pi\bigg( \frac x{\log x} \bigg) \li(x) \bigg) + \int_{x/\log x}^{x-\sqrt x} \frac{\pi(t)}{\log(x-t)} \,dt \\
  &= \int_{x/\log x}^{x-\sqrt x} \frac{\pi(t)}{\log(x-t)} \,dt + O\bigg( \frac{x^2}{\log^3x} \bigg),
\end{align*}
and so equation~\eqref{li sum} becomes
\[
Q(x) = \int_{x/\log x}^{x-\sqrt x} \frac{\pi(t)}{\log(x-t)} \,dt + O\bigg( \frac{x^2}{\log^3x} \bigg).
\]
Using the prime number theorem again, this becomes
\begin{align}
  Q(x) &= \int_{x/\log x}^{x-\sqrt x} \frac1{\log(x-t)}\bigg( \frac t{\log t} + O\bigg( \frac t{\log^2t} \bigg) \bigg) \,dt + O\bigg( \frac{x^2}{\log^3x} \bigg) \notag \\
  &= \int_{x/\log x}^{x-\sqrt x} \frac t{(\log t)\log(x-t)} \,dt + O\bigg( \int_{x/\log x}^{x-\sqrt x} \frac t{(\log^2t) \log(x-t)}\,dt + \frac{x^2}{\log^3x} \bigg).
  \label{two integrals}
\end{align}
In the error term, again $\log(x-t) \gg \log x$ and $\log^2t \gg \log^2x$ due to the endpoints of integration, and so the entire integral is $\ll x^2/\log^3x$. In the main term, we have
\[
\log x \ge \log t \ge \log \frac x{\log x} = \log x - \log\log x = (\log x) \bigg( 1+O\bigg( \frac{\log\log x}{\log x} \bigg) \bigg),
\]
and therefore equation~\eqref {two integrals} becomes
\begin{equation}
  Q(x) = \frac1{\log x}\bigg( 1+O\bigg( \frac{\log\log x}{\log x} \bigg) \bigg) \int_{x/\log x}^{x-\sqrt x} \frac t{\log(x-t)} \,dt + O\bigg( \frac{x^2}{\log^3x} \bigg).
  \label{one integral}
\end{equation}
Finally,
\begin{align}
  \int_{x/\log x}^{x-\sqrt x} \frac t{\log(x-t)} \,dt &= \int_0^{x-2} \frac t{\log(x-t)} \,dt + O\bigg( \int_0^{x/\log x} t \,dt + \int_{x-\sqrt x}^{x-2} t \,dt \bigg) \notag \\
  &= \int_2^x \frac{x-u}{\log u} \,du + O\bigg( \frac{x^2}{\log^2x} \bigg) \notag \\
  &= x\li(x) - \int_2^x \frac u{\log u}\,du + O\bigg( \frac{x^2}{\log^2x} \bigg).
  \label{strange li}
\end{align}
By integration by parts, this integral is
\begin{align*}
  \int_2^x \frac u{\log u}\,du &= \frac{u^2}2 \frac1{\log u} \bigg|_2^x + \int_2^x \frac{u^2}2 \frac1{u\log^2 u} \,du \\
  &= \frac{x^2}{2\log x} + O\bigg( 1 + \int_2^{\sqrt x} \frac u{\log^2 u}\,du + \int_{\sqrt x}^x \frac u{\log^2u} \,du \bigg) \\
  &= \frac{x^2}{2\log x} + O\bigg( \sqrt x \cdot x + x \frac x{\log^2 x} \bigg) = \frac{x^2}{2\log x} + O\bigg( \frac{x^2}{\log^2 x} \bigg).
\end{align*}
Therefore equation~\eqref{strange li} becomes
\[
\int_{x/\log x}^{x-\sqrt x} \frac t{\log(x-t)} \,dt = x\li(x) - \frac{x^2}{2\log x} + O\bigg( \frac{x^2}{\log^2x} \bigg) = \frac{x^2}{2\log x} + O\bigg( \frac{x^2}{\log^2x} \bigg)
\]
by the fact that $\li(x) = x/\log x + O(x/\log^2 x)$. Using this in equation~\eqref{one integral} finally yields
\begin{align*}
  Q(x) &= \frac1{\log x}\bigg( 1+O\bigg( \frac{\log\log x}{\log x} \bigg) \bigg) \bigg( \frac{x^2}{2\log x} + O\bigg( \frac{x^2}{\log^2x} \bigg) \bigg) + O\bigg( \frac{x^2}{\log^3x} \bigg) \\
  &= \frac{x^2}{2\log^2 x} + O\bigg( \frac{x^2\log\log x}{\log^3x} \bigg),
\end{align*}
as claimed.
\end{proof}

Equipped with Lemma \ref{Q lemma}, we are now prepared to prove Theorem~\ref{UnconditionalABound}.

\begin{proof}[Proof of Theorem \ref{UnconditionalABound}]
Starting with the definitions of $a(m)$ and $A(M)$, we have
\begin{equation*}
A(M) = \sum_{m=1}^{2M} a(m) = \sum_{m=1}^{2M} \sum_{d\mid m} R(d) = \sum_{m=1}^{2M} \sum_{d\mid m} \sum_{p+q=d} 1 = \sum_{p+q \le 2M} \sum_{\substack{1\le m\le 2M \\ (p+q)\mid m}} 1.
\end{equation*}
Writing $m=(p+q)n$, we obtain
\begin{equation}
A(M) = \sum_{p+q \le 2M} \sum_{1\le n\le 2M/(p+q)} 1 = \sum_{1\le n\le M/2} \sum_{p+q \le 2M/n} 1 = \sum_{1\le n\le M/2} Q\bigg( \frac{2M}{p+q} \bigg).
\label{telescoped}
\end{equation}
The trivial bound $Q(x) \le x^2$ allows us to write
\[
A(M) = \sum_{1\le n\le \log^3M} Q\bigg( \frac{2M}n \bigg) + O\bigg( \sum_{n > \log^3M} \bigg( \frac{2M}n \bigg)^2 \bigg) = \sum_{1\le n\le \log^3M} Q\bigg( \frac{2M}n \bigg) + O\bigg( \frac{M^2}{\log^3M} \bigg),
\]
since $\sum_{n>\log^3M} n^{-2} \ll 1/\log^3M$ by comparison with an integral. We use Lemma~\ref{Q lemma} to get
\begin{align*}
A(M) &= \sum_{1\le n\le \log^3M} \bigg( \frac{(2M/n)^2}{2\log^2(2M/n)} + O\bigg( \frac{(2M/n)^2\log\log(2M/n)}{\log^3(2M/n)} \bigg) \bigg) + O\bigg( \frac{M^2}{\log^3M} \bigg) \\
&= 2M^2 \sum_{1\le n\le \log^3M} \frac1{\log^2(2M/n)}\frac1{n^2} + O\bigg( \sum_{1\le n\le\log^3M} \frac{\sqrt{2M}\log\log2M}{\log^32M} \bigg( \frac{2M}n \bigg)^{3/2} + \frac{M^2}{\log^3M} \bigg),
\end{align*}
since $\sqrt x\log\log x/\log^3x$ is an (eventually) increasing function of~$x$. By the convergence of $\sum_n n^{-3/2}$, we obtain
\[
A(M) = 2M^2 \sum_{1\le n\le \log^3M} \frac1{\log^2(2M/n)}\frac1{n^2} + O\bigg( \frac{M^2\log\log M}{\log^3M} \bigg).
\]
Finally, we have $\log(2M/n) = \log M - \log(n/2) = \log M + O(\log(\log^3M)) = (\log M)(1 + O(\log\log M/\log M))$ as before. Therefore
\[
A(M) = \frac{2M^2}{\log^2M} \bigg(1 + O\bigg( \frac{\log\log M}{\log M} \bigg) \bigg) \sum_{1\le n\le \log^3M} \frac1{n^2} + O\bigg( \frac{M^2\log\log M}{\log^3M} \bigg).
\]
We conclude that
\begin{align*}
A(M) &= \frac{2M^2}{\log^2M} \bigg(1 + O\bigg( \frac{\log\log M}{\log M} \bigg) \bigg) \bigg( \zeta(2) + O\bigg( \frac1{\log^3M} \bigg) \bigg) + O\bigg( \frac{M^2\log\log M}{\log^3M} \bigg) \\
&= \frac{\pi^2M^2}{3\log^2M} + O\bigg( \frac{M^2\log\log M}{\log^3M} \bigg),
\end{align*}
as desired.
\end{proof}

\end{document}